\documentclass[12pt, a4paper, reqno]{amsart}

\usepackage{amssymb}
\usepackage[english]{babel}
\usepackage{blindtext}
\usepackage[tableposition=top]{caption}
\usepackage{comment}
\usepackage{csquotes}
\usepackage[inline]{enumitem}
\usepackage[T1]{fontenc}
\usepackage{mathtools}
\usepackage{physics}
\usepackage{tikz-cd}

\usepackage{hyperref}
\usepackage[capitalize]{cleveref}
\hypersetup{
    colorlinks=true,
    allcolors=black,
}

\setlist[enumerate]{topsep=2pt,itemsep=-0.5ex,partopsep=0ex,parsep=1ex}

\theoremstyle{plain}

\newtheorem{theorem}{Theorem}[section]

\newtheorem{lemma}[theorem]{Lemma}

\newtheorem*{question}{Question}

\theoremstyle{definition}
\newtheorem{definition}[theorem]{Definition}

\theoremstyle{remark}
\newtheorem{remark}[theorem]{Remark}
\newtheorem{remarks}[theorem]{Remarks}
\newtheorem{example}[theorem]{Example}

\newtheorem*{claim}{Claim}

\title{Pre-Lagrangian tori transverse to an Anosov flow}

\author{Francesco Ruscelli}
\address{Mathematical Institute, University of Heidelberg, 69120 Heidelberg, Germany}
\email{fruscelli@mathi.uni-heidelberg.de}
\begin{document}

\begin{abstract}
    An Anosov flow on a smooth three-manifold \( M \) gives rise to a Liouville structure on \( \mathbb{R} \times M \) by a construction of Mitsumatsu. In a recent paper, Cieliebak, Lazarev, Massoni and Moreno ask whether an embedded torus \( \Sigma \subseteq M \) transverse to an Anosov flow gives rise to a Lagrangian in \( \mathbb{R} \times M \). We show that the answer to this question is in general negative by finding a topological obstruction related to the foliations induced on the torus by the weak stable and unstable bundles of the flow. Going in the opposite direction, we show that the answer is positive if the induced foliations do not admit parallel compact leaves.
\end{abstract}

\maketitle

\section{Introduction}
\label{sec:introduction}
Given a smooth Anosov flow \( \varphi_t \colon M \to M \) on a closed, oriented three-manifold \( M \), Mitsumatsu \cite{mitsumatsu1995} showed that there exists a pair \( \xi_{\pm} \) of oppositely oriented contact structures on \( M \) intersecting transversely along the direction of the flow (see also the work of Eliashberg and Thurston \cite{eliashberg1998}). The \textit{bicontact structure} \( (\xi_+, \xi_-) \) can be described by a \textit{Liouville pair}, i.e.\ a pair of contact forms \( \alpha_{\pm} \in \Omega^1(M) \) for \( \xi_{\pm} \) such that \( \lambda = e^s \alpha_+ + e^{-s} \alpha_- \) is a Liouville form on \( \mathbb{R}_s \times M \), that is \( d \lambda \) is symplectic. This Liouville structure only depends, up to exact symplectomorphism, on the homotopy class of the Anosov flow \cite[Corollary 1]{massoni2023}. The map
\begin{equation*}
    \{ \text{Anosov flows} \} \to \{ \text{Bicontact structures} \}
\end{equation*}
was upgraded to a full correspondence by Hozoori \cite{hozoori2022} and \cite{massoni2023}. They showed that Anosov flows (up to positive time reparameterization) correspond to special pairs of contact forms, which we call \textit{Anosov-Liouville pairs}. The precise definition will be given in \cref{sec:preliminaries}.

Given this picture, it is natural to wonder whether one can leverage tools from symplectic topology to extract information about Anosov flows, which is a direction first explored recently by Cieliebak, Lazarev, Massoni and Moreno \cite{cieliebak2022}. Amongst many things, the authors investigate the existence of closed Lagrangians in \( \mathbb{R} \times M \) and they study two particular examples, namely geodesic flows on hyperbolic surfaces and suspensions of hyperbolic toral automorphisms. They also pose the following question.
\begin{question}[{\cite[Question 7.4]{cieliebak2022}}]
    Does a torus transverse to an Anosov flow give rise to a Lagrangian?
\end{question}
The question, as formulated by the authors, is somewhat vague as it does not specify how we are supposed to embed a torus \( \Sigma \subseteq M \) in \( \mathbb{R} \times M \). We will make things more precise soon.

The motivation behind this question is threefold. Firstly, we do have an example of such tori, namely the \( \mathbb{T}^2 \)-fibers of suspension flows. Secondly, there is an abundance of Anosov flows admitting transverse tori. A special subclass which is of particular interest is that of \textit{nontransitive Anosov flows}, i.e.\ flows that do not admit any dense forward orbit. These admit a collection of transverse tori by a classical result of Brunella \cite{brunella1993}. While their dynamical relevance is well-known, it is not clear whether they carry any symplectic information. Lastly, Brunella also observed that a torus transverse to an Anosov flow is incompressible. Thus, if a transverse torus \( \Sigma \cong \{ 0 \} \times \Sigma \subseteq (\mathbb{R} \times M, d \lambda) \) is Lagrangian, it is automatically weakly exact (recall that a Lagrangian submanifold \( L \subseteq (V, \omega) \) is called \textit{weakly exact} if \( \omega|_{\pi_2(V, L)} = 0 \)). Any such torus defines an object in the non-exact Fukaya category of \( \mathbb{R} \times M \), which is a symplectic invariant of the flow introduced in \cite{cieliebak2022}.

\smallskip

The most natural way to embed a torus \( \Sigma \subseteq M \) in \( \mathbb{R} \times M \) is as the graph of a function. Thus, for the sake of this paper, we give the following nonstandard definition.
\begin{definition}\label{def:pre-lagrangian}
    Let \( \varphi_t \colon M \to M \) be a smooth Anosov flow.
    An embedded surface \( \Sigma \subseteq M \) is \textit{pre-Lagrangian} if there exist an Anosov-Liouville pair \( (\alpha_+, \alpha_-) \in \Omega^1(M) \times \Omega^1(M) \) supporting \( \varphi_t \) and a function \( f \in C^{\infty}(\Sigma) \) such that
    \begin{equation*}
        \Gamma(f) = \{ (f(x), x) \in \mathbb{R} \times \Sigma \}
    \end{equation*}
    is Lagrangian in \( (\mathbb{R} \times M, d \lambda) \), where \( \lambda = e^s \alpha_+ + e^{-s} \alpha_- \).
\end{definition}

We provide a necessary condition for a transverse torus to be pre-Lagrangian. The obstruction comes from the foliations induced on the torus by the stable and unstable bundles of the flow. More precisely, let
\begin{equation*}
    TM = \langle X \rangle \oplus E^s \oplus E^u
\end{equation*}
be the \( \varphi_t \)-invariant hyperbolic splitting, \( X \) being the generator of the flow, and let
\begin{equation*}
    \begin{aligned}
        E^{ws} &= \langle X \rangle \oplus E^s , \\
        E^{wu} &= \langle X \rangle \oplus E^u
    \end{aligned}
\end{equation*}
be the weak stable and weak unstable bundles respectively. Let now \( \Sigma \subseteq M \) be an embedded torus transverse to \( X \). Since \( E^{ws} \) and \( E^{wu} \) contain \( X \), they induce \( 1 \)-dimensional (oriented) foliations \( \mathcal{F}^{ws}_{\Sigma} \) and \( \mathcal{F}^{wu}_{\Sigma} \) on \( \Sigma \). Identifying \( \Sigma \) with \( \mathbb{T}^2 \) via some diffeomorphism \( F \colon \Sigma \to \mathbb{T}^2 \), we view oriented foliations on \( \Sigma \) as maps
\begin{equation*}
    g \colon \mathbb{T}^2 \to \mathbb{R}^2 \setminus \{ 0 \} \simeq S^1,
\end{equation*}
which obviously depend on \( F \). Their homotopy class is completely characterized by the map \( g_* \colon \pi_1(\mathbb{T}^2) \to \pi_1(S^1) \).

\begin{theorem}\label{thm:obstruction}
    Let \( \Sigma \subseteq M \) be an embedded torus transverse to an Anosov flow \( \varphi_t \) on \( M \). Let \( g^{ws}, g^{wu} \colon \mathbb{T}^2 \to S^1 \) be the maps corresponding to the foliations \( \mathcal{F}^{ws}_{\Sigma}, \mathcal{F}^{wu}_{\Sigma} \) respectively. If \( \Sigma \) is pre-Lagrangian, then the maps \( g^{ws}_*, g^{wu}_* \colon \pi_1(\mathbb{T}^2) \to \pi_1(S^1) \) are trivial.
\end{theorem}

\begin{remarks} The following remarks are in order.
    \begin{enumerate}
        \item The maps \( g^{ws} \) and \( g^{wu} \) depend on the identification \( \Sigma \cong \mathbb{T}^2 \) we have chosen. However, it is implicit in the statement of the theorem that the fact that they induce the trivial map on \( \pi_1 \) does not.
        \item Since \( \mathcal{F}^{ws}_{\Sigma} \)  and \( \mathcal{F}^{wu}_{\Sigma} \) are transverse, the maps \( g^{ws} \) and \( g^{wu} \) induce the same map on \( \pi_1 \).
        \item The condition in \cref{thm:obstruction} is not sufficient. In a later section we will exhibit an Anosov flow admitting a transverse torus foliated by a homotopically trivial foliation which is, however, not pre-Lagrangian. This has to do with the fact that the induced maps on \( \pi_1 \) do not distinguish between isotopy classes of foliations, whereas being pre-Lagrangian does.
        \item \cref{thm:obstruction} shows that the answer to \cite[Question 7.4]{cieliebak2022} is in general negative. We can for instance consider the first example of a nontransitive Anosov flow, constructed by Franks and Williams in 1980 \cite{franks1980}. This flow has exactly one (up to isotopy) distinguished torus. The foliations induced on it are depicted in \cref{fig:franks-williams}. Since they are clearly not trivial on \( \pi_1 \), this torus is not pre-Lagrangian.
            \begin{figure}[h!]
                \centering
                \includegraphics[scale=1.1]{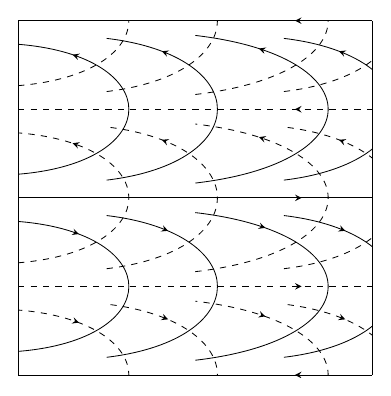}
                \caption{The foliations \( \mathcal{F}^{ws}_{\Sigma}, \mathcal{F}^{ws}_{\Sigma} \) in the Franks-Williams flow.}
                \label{fig:franks-williams}
            \end{figure}
    \end{enumerate}
\end{remarks}

Although, as noted above, the converse of \cref{thm:obstruction} is in general false, there are still cases in which something can be said. To make this precise, we borrow the following definition from Bonatti and Zhang.

\begin{definition}[{\cite[Def.\ 1.1]{bonatti2016}}]\label{def:parallel_compact}
    Two transverse foliations \( \mathcal{F} \) and \( \mathcal{G} \) on a torus have \textit{parallel compact leaves} if there exists a compact leaf of \( \mathcal{F} \) and a compact leaf of \( \mathcal{G} \) that are in the same free homotopy class.
\end{definition}

\begin{theorem}\label{thm:existence}
    Let \( \Sigma \subseteq M \) be an embedded torus transverse to an Anosov flow \( \varphi_t \) on \( M \) and assume that the foliations \( \mathcal{F}^{ws}_{\Sigma}, \mathcal{F}^{wu}_{\Sigma} \) do not have parallel compact leaves. Then, \( \Sigma \) is pre-Lagrangian.
\end{theorem}

A very special case of this, for which the statement can be explicitly checked, is the \( \mathbb{T}^2 \)-fibers of suspension flows. There, the induced foliations are affine.

In \cref{sec:preliminaries} we cover the necessary background and basic definitions we are going to need, both from Anosov dynamics and symplectic geometry. We also discuss foliations on the torus, specifically focusing on the ones defined by nowhere-vanishing closed \( 1 \)-forms, which play a prominent role in our work. In \cref{sec:main_proof} we prove \cref{thm:obstruction} and we show that the condition in the theorem is not sufficient by exhibiting an explicit counterexample. Finally, in \cref{sec:examples}, we give a proof of \cref{thm:existence}.

\subsection{Acknowledgments}
I would like to thank Peter Albers, Tom Stalljohann and Tobias Witt for pointing out a mistake and for helpful discussions. I am also grateful to Kai Cieliebak for his feedback and input.
This research was funded by the Deutsche Forschungsgemeinschaft (DFG, German Research Foundation) – 281869850 (RTG 2229).

\section{Preliminaries}\label{sec:preliminaries}
In what follows, \( M \) is a closed, oriented three-manifold and any flow we consider is assumed to be smooth.

\subsection{Anosov flows}
In this section we recall the main definitions and facts about Anosov dynamics we are going to need in later sections.

\begin{definition}
    A nonsingular flow \( \varphi_t \colon M \to M \) is \textit{Anosov} if there exists a continuous \( \varphi_t \)-invariant splitting
    \begin{equation*}
        TM = \langle X \rangle \oplus E^s \oplus E^u
    \end{equation*}
    into line bundles and constants \( C \geq 1, \lambda > 1 \) such that:
    \begin{enumerate}
        \item \( X = \dot{\varphi_t} \) is the generator of the flow;
        \item the flow contracts (resp.\ expands) exponentially \( E^s \) (resp.\ \( E^u \)):
            \begin{equation*}
                \begin{aligned}
                    \norm{d \varphi_t |_{E^s}} &< C\lambda^t, \\
                    \norm{d \varphi_{t} |_{E^u}} &> C \lambda^{t}
                \end{aligned}
            \end{equation*}
            for all \( t > 0 \). \label{itm:hyperbolic_set}
    \end{enumerate}
\end{definition}

\begin{remark}
    The metric with respect to which we compute the norm of \( d \varphi_t \) is not relevant. Any other metric will do (up to changing the constants), given that \( M \) is closed by assumption.
\end{remark}

The topological dynamics of Anosov flows is completely described by the celebrated \textit{spectral decomposition theorem} proven by Smale in 1967 \cite{smale1967}. In order to state it, let us briefly recall some notions of recurrence.

\begin{definition} Let \( \varphi_t \colon M \to M \) be a flow.
    \begin{enumerate}
        \item A point \( p \in M \) is \textit{nonwandering} if for every neighborhood \( U \) of \( p \) and every \( T > 0 \) there exists \( t > T \) such that \( U \cap \varphi_t(U) \neq \emptyset \).
        \item \( \varphi_t \) is \textit{transitive} if it admits a dense forward orbit.
    \end{enumerate}
\end{definition}
Note that the set \( \Omega(\varphi_t) \) of nonwandering points is compact and \( \varphi_t \)-invariant. Moreover if \( \varphi_t \) is transitive, then \( \Omega(\varphi_t) = M \).

In general, the topological dynamics of a flow always have a recurrent and a transient part, which can be seen using standard results such as the existence of Lyapunov functions. Smale's spectral decomposition theorem gives a more precise description of the recurrent dynamics of the flow, under the assumption that it is hyperbolic.
We will state the theorem only for Anosov flows. While it holds under weaker assumptions, we will not need that degree of generality here. A thorough presentation can be found in \cite{smale1967}.

\begin{theorem}[Smale's spectral decomposition theorem]
    Let \( \varphi_t \colon M \to M \) be an Anosov flow. Then, \( \Omega(\varphi_t) \) is the disjoint union
    \begin{equation*}
        \Omega(\varphi_t) = \coprod_{j = 1}^k \Lambda_j
    \end{equation*}
    of compact, locally maximal, \( \varphi_t \)-invariant subsets \( \Lambda_j \) on which the flow is transitive.
\end{theorem}

A subset \( \Lambda \subseteq M \) is \textit{locally maximal} if it admits a neighborhood \( U \) such that \( \bigcap_{t \in \mathbb{R}} \varphi_t(U) = \Lambda \). The \( \Lambda_j \)'s in the above proposition are usually called \textit{basic sets} in the literature. An obvious consequence of the theorem is that \( \varphi_t \) is transitive if and only if its spectral decomposition is \( \{ M \} \).

In the case of nontransitive flows, the spectral decomposition theorem comes with an associated topological decomposition, as the following result shows.

\begin{theorem}[\cite{brunella1993}]
    Let \( \varphi_t \colon M \to M \) be a nontransitive Anosov flow with spectral decomposition \( \Omega(\varphi_t) = \coprod_{j = 1}^k \Lambda_j \). Then, there exists a finite collection of pairwise disjoint, pairwise nonisotopic, incompressible embedded tori \( \{ \Sigma_i \}_{i = 1, \dots, n} \) such that
    \begin{enumerate}
        \item \( \bigcup_{i = 1}^n \Sigma_j \cap \Omega(\varphi_t) = \emptyset \),
        \item each connected component of \( M \setminus \bigcup_{i = 1}^n \Sigma_j \) contains exactly one the \( \Lambda_j \)'s.
    \end{enumerate}
\end{theorem}
What the theorem says is that there exists a collection of distinguished tori that \enquote{separate} the basic sets. Moreover, from the proof one sees that these tori are transverse to the flow, as they arise as regular level sets of a Lyapunov function for \( \varphi_t \).

\subsection{Anosov-Liouville pairs}\label{subsec:al}
In this section we give introduce Anosov-Liouville pairs, which provide a contact topological framework for working with Anosov flows. The original idea behind this construction is due to Mitsumatsu \cite{mitsumatsu1995}. We follow the presentation in \cite{massoni2023} (see also \cite{eliashberg1998}, \cite{hozoori2022}).

\begin{definition}\label{def:liouville}
    \mbox{}
    \begin{enumerate}
        \item
            A pair \( (\alpha_+, \alpha_-) \in \Omega^{1}(M) \times \Omega^{1}(M) \) is called a \textit{Liouville pair} if
            \begin{equation*}
                \lambda = e^s \alpha_+ + e^{-s}\alpha_-
            \end{equation*}
            defines a Liouville form on \( \mathbb{R}_s \times M \).
        \item \( (\alpha_+, \alpha_-) \) is called an \textit{Anosov-Liouville pair} if both \( (\alpha_+, \alpha_-) \) and \( (\alpha_+, -\alpha_-) \) are Liouville pairs.
    \end{enumerate}

\end{definition}

\begin{remarks}
    We note the following:
    \begin{enumerate}
        \item from the above definition, one easily sees that if \( (\alpha_+, \alpha_-) \) is an Anosov-Liouville pair, then \( \alpha_+ \) (resp.\ \( \alpha_- \)) is a positive (resp.\ negative) contact form and the contact structures \( \xi_{\pm} = \mathop{\mathrm{ker}} \alpha_{\pm} \) are transverse. This is what is usually called a \textit{bicontact structure} in the literature. We say that \( (\alpha_+, \alpha_-) \) (and its associated bicontact structure) \textit{supports} \( \varphi_t \) if \( \xi_+ \cap \xi_- = \langle X \rangle \), where \( X \in \mathfrak{X}(M) \) generates the flow.
        \item Not all bicontact structures arise from Anosov-Liouville pairs \cite{massoni2023} (see also \cite{eliashberg1998}).
        \item \( C^{\infty}(M) \) acts on the space of Anosov-Liouville pairs supporting a given flow by
            \begin{equation*}
                f \cdot (\alpha_+, \alpha_-) = (e^f \alpha_+, e^{-f} \alpha_-).
            \end{equation*}
    \end{enumerate}
\end{remarks}

Here is a simple characterization of Anosov-Liouville pairs.
\begin{lemma}[{\cite[Lemma 2.7]{massoni2023}}]\label{lemma:al_characterization}
    Let \( \operatorname{dvol} \) be any volume form on \( M \) and \( \alpha_{\pm} \in \Omega^1(M) \). Write
    \begin{equation*}
        \begin{aligned}
            &\alpha_+ \wedge d \alpha_+ = f_+ \operatorname{dvol}, \\
            &\alpha_- \wedge d \alpha_- = -f_- \operatorname{dvol}, \\
            &d(\alpha_- \wedge \alpha_+) = f_0 \operatorname{dvol}
        \end{aligned}
    \end{equation*}
    for smooth functions \( f_{\pm}, f_0 \colon M \to \mathbb{R} \). Then, \( (\alpha_+, \alpha_-) \) is Anosov-Liouville if and only if \( f_{\pm} > 0 \) and \( f_0^2 < 4 f_+ f_- \).
\end{lemma}

Let us now consider an Anosov flow \( \varphi_t \) and a generator \( X \in \mathfrak{X}(M) \). We have a \( \varphi_t \)-invariant splitting
\begin{equation*}
    TM = \langle X \rangle \oplus E^s \oplus E^u,
\end{equation*}
with the flow exponentially contracting (resp.\ expanding) \( E^s \) (resp.\ \( E^u \)).
We will assume that \( E^s \) and \( E^u \) are orientable, which can always be arranged by passing to a double cover.
Since \( \varphi_t \) is smooth, the weak stable and weak unstable bundles
\begin{equation*}
    \begin{aligned}
        E^{ws} &= E^s \oplus \langle X \rangle, \\
        E^{wu} &= E^u \oplus \langle X \rangle
    \end{aligned}
\end{equation*}
are \( C^1 \) by \cite[Corollary 1.8]{hasselblatt1994}. There exist \( C^1 \) 1-forms \( \alpha_u \) and \( \alpha_s \) such that \( \mathop{\mathrm{ker}} \alpha_s = E^{wu} \), \( \mathop{\mathrm{ker}} \alpha_u = E^{wu} \) and
\begin{equation*}
    \begin{aligned}
        \mathcal{L}_X \alpha_u &= r_u \alpha_u, \\
        \mathcal{L}_X \alpha_s &= r_s \alpha_s,
    \end{aligned}
\end{equation*}
where \( r_u > 0 \) and \( r_s < 0 \) are \( C^1 \) functions on \( M \) (for the exact details see \cite{hozoori2022} and \cite{massoni2023}). Then, one sees that \( \alpha_{\pm} = \alpha_u \mp \alpha_s \) (and any \( C^{\infty} \)-close smoothing thereof) form an Anosov-Liouville pair with the contact structures \( \xi_{\pm} = \mathop{\mathrm{ker}} \alpha_{\pm} \) intersecting along the direction \( X \) of the flow. The pair \( (\alpha_u, \alpha_s) \) is called a \textit{defining pair} for \( \varphi_t \) and we say that \( (\alpha_+, \alpha_-) \) is a \textit{standard Anosov-Liouville pair} supporting \( \varphi_t \).

The relation between Anosov flows and Anosov-Liouville pairs is clarified by the following result.
\begin{theorem}[{\cite[Theorem 1]{massoni2023}}]
    Let \( \varphi_t \colon M \to M \) be a smooth nonsingular flow. Then, \( \varphi_t \) is Anosov if and only there exists an Anosov-Liouville pair \( (\alpha_+, \alpha_-) \) supporting it.
\end{theorem}

Moreover, the space of Anosov-Liouville pairs associated to an Anosov flow is not too bad. Indeed, we have:

\begin{theorem}[{\cite[Theorem 2]{massoni2023}}]\label{thm:contractible}
    Let \( \varphi_t \) be an Anosov flow on \( M \) as above. Then:
    \begin{enumerate}
        \item the space of standard Anosov-Liouville pairs supporting \( \varphi_t \) is convex, hence contractible.
        \item The space of Anosov-Liouville pairs supporting \( \varphi_t \) is contractible.
    \end{enumerate}
\end{theorem}
This result can be used to show that the Liouville structure on \( \mathbb{R} \times M \) given by the Liouville form \( \lambda = e^s \alpha_+ + e^{-s} \alpha_- \) only depends on the homotopy class of the Anosov flow up to exact symplectomorphism. It is conjectured that it only depends on the orbit equivalence class of the flow.

\begin{example}\label{ex:examples} These examples historically appeared in the context of searching for non-Weinstein Liouville domains. We refer the reader to \cite{massoni2023} for the details.
    \begin{enumerate}
        \item This construction is due to McDuff \cite{mcduff1991}. Let \( \Sigma \) be a closed hyperbolic surface and let \( \sigma \in \Omega^2(\Sigma) \) be the area form associated to a hyperbolic metric on \( \Sigma \) with total area \( -\chi(\Sigma) \). Consider the twisted symplectic form
            \begin{equation*}
                \omega_{\sigma} = \omega_{\mathrm{can}} + \pi ^* \sigma,
            \end{equation*}
            on \( T^* \Sigma \), where \( \omega_{\mathrm{can}} = d \lambda_{\mathrm{can}} \) is the canonical symplectic form on the cotangent bundle and \( \pi \colon T^* \Sigma \to \Sigma \) is the projection. Let us denote the unit cotangent bundle of \( \Sigma \) by \( S^* \Sigma \). Since \( \pi^* \sigma \) is exact when restricted to \( S^* \Sigma \), we can write \( \pi^* \sigma|_{S^* \Sigma} = d \lambda_{{\mathrm{pre}}} \) for some \( \lambda_{\mathrm{pre}} \in \Omega^1(S^* \Sigma) \). As the name suggests, \( \lambda_{\mathrm{pre}} \) is a prequantization form for the \( S^1 \)-bundle \( S^* \Sigma \to \Sigma \). Then, it is not too hard to show that \( (\lambda_{\mathrm{can}} |_{S* \Sigma}, \lambda_{\mathrm{pre}}) \) is an Anosov-Liouville pair. The underlying Anosov flow is smoothly conjugate to the geodesic flow on \( S \Sigma \).
        \item The following class of examples is due to Mitsumatsu \cite{mitsumatsu1995} and, independently, Geiges \cite{geiges1995}. On \( \mathbb{R}^3 \) consider the \( 1 \)-forms
            \begin{equation*}
                \alpha_{\pm} = \pm e^{z} dx + e^{-z} dy.
            \end{equation*}
            It is straightforward to check that \( (\alpha_+, \alpha_-) \) is an Anosov-Liouville pair. Let now \( A \in \operatorname{SL} (2, \mathbb{Z}) \) be a hyperbolic element and conjugate it to a diagonal matrix via an element \( P \in \operatorname{SL}(2, \mathbb{R}) \):
            \begin{equation*}
                PAP^{-1} =
                \begin{pmatrix}
                    e^\nu & 0 \\
                    0     & e^{-\nu}
                \end{pmatrix}
                = D_{\nu},
            \end{equation*}
            where \( \nu > 0 \).
            Consider the map
            \begin{equation*}
                \begin{aligned}
                    \psi \colon \mathbb{R}^2 &/ P(\mathbb{Z}^2) \times \mathbb{R} \to \mathbb{R}^2 / P(\mathbb{Z}^2) \times \mathbb{R} \\
                    & (x, z) \mapsto (D_{\nu} x, z - \nu).
                \end{aligned}
            \end{equation*}
            Since the forms \( \alpha_\pm \) are invariant under \( \psi \), they induce an Anosov-Liouville pair on the mapping torus \( \big( \mathbb{R}^2 / P(\mathbb{Z}^2) \times \mathbb{R} \big) / \psi \). The underlying vector field is conjugate via \( P \times \operatorname{id} \) to the suspension of \( A \) seen as a hyperbolic automorphism of the torus \( \mathbb{T}^2 = \mathbb{R}^2 / \mathbb{Z}^2 \). \label{itm:ii}
    \end{enumerate}
\end{example}

\subsection{Foliations on \texorpdfstring{\( \mathbb{T}^2 \)}{the torus}}
Here we prove a couple of useful facts about (nonsingular) foliations on \( \mathbb{T}^2 = \mathbb{R}^2 / \mathbb{Z}^2 \). We will only be interested in oriented foliations of class at least \( C^1 \).

\begin{example}
    Let \( X = a \partial_x + b \partial_y \) be a constant vector field on \( \mathbb{T}^2 \). Then, \( X \) induces a foliation of \( \mathbb{T}^2 \) by either embedded copies of \( S^1 \) or immersed copies of \( \mathbb{R} \), depending on whether the slope of \( X \) is rational or not.

    More generally, let \( \varphi \colon S^1 \to S^1 \) be an orientation-preserving diffeomorphism of the circle and consider its \textit{suspension}, namely the manifold \( M_{\varphi} = S^1 \times [0, 1] / {\sim} \), where we identify \( (x, 1) \sim \big( \varphi(x), 0 \big) \). Since \( \varphi \) is isotopic to the identity, \( M_{\varphi} \) is diffeomorphic to \( \mathbb{T}^2 \). Moreover, if \( t \) denotes the coordinate on \( [0, 1] \), the flow of the vertical vector field \( \partial_t \) factors through to the quotient and gives a well-defined flow (hence a foliation) on \( M_{\varphi} \). We denote this foliation by \( \operatorname{Susp}(\varphi) \).
\end{example}

There are also other types of foliations. To introduce them, let us give the following definition.
\begin{definition}
    We say that a foliation \( \mathcal{F} \) admits a \textit{Reeb component} (or \textit{Reeb annulus}) if there is a compact annulus bounded by two oppositely oriented leaves of \( \mathcal{F} \) whose interior does not contain any compact leaves.

    \begin{figure}[h]
        \centering
        \includegraphics[scale=1]{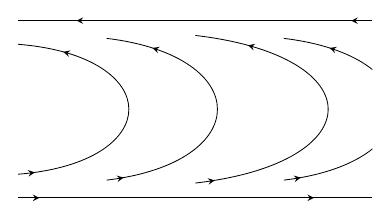}
        \caption{A Reeb annulus.}
        \label{fig:reeb}
    \end{figure}
\end{definition}

The presence of Reeb annuli is the only phenomenon that prevents a foliation from being a suspension, as the following theorem shows. We refer the reader to \cite{hector1981} for a proof.
\begin{theorem}
    Let \( \mathcal{F} \) be a foliation on \( \mathbb{T}^2 \) of class \( C^k \) for \( k \geq 1 \). Then, either
    \begin{enumerate}
        \item \( \mathcal{F} \) admits a Reeb component, or
        \item \( \mathcal{F} \) is the suspension of a diffeomorphism \( \varphi \) of \( S^1 \), i.e.\ there exists a \( C^k \) diffeomorphism \( f \colon M_{\varphi} \to \mathbb{T}^2 \) such that \( f\big( \operatorname{Susp}(\varphi) \big) = \mathcal{F} \).
    \end{enumerate}
\end{theorem}

For our purposes, it will be convenient to view foliations in a different guise, namely if \( \mathcal{F} \) is a foliation on the torus, we view it as a map \( g \colon \mathbb{T}^2 \to \mathbb{R}^2 \setminus \{ 0 \} \simeq S^1 \) by picking a nonvanishing vector field directing it. Then, the path component of \( \mathcal{F} \) in the space of foliations is determined by \( g_* \colon \pi_1(\mathbb{T}^2) \to \pi_1(S^1) \). In other words, two foliations can be homotoped one to the other through foliations if and only if they induce the same map on \( \pi_1 \). This follows easily from the canonical correspondence
\begin{equation*}
    \begin{aligned}
        [\mathbb{T}^2&, S^1] \to H^1(\mathbb{T}^2; \mathbb{Z}) \\
        &[g] \mapsto g^* \theta,
    \end{aligned}
\end{equation*}
where \( \theta \in H^1(S^1) \) is defined by \( \langle \theta, [S^1] \rangle = 1 \).

\begin{remark}\label{rmk:trivialization}
    Here, we are implicitly using the canonical trivialization of \( T(\mathbb{T}^2) = \mathbb{T}^2 \times \mathbb{R}^2 \). If we want to talk about foliations on an abstract torus \( \Sigma \) in the same fashion, we have to deal with an additional headache, namely the fact that there is no canonical trivialization of \( T \Sigma \). If \( F \colon \Sigma \to \mathbb{T}^2 \) is an orientation-preserving diffeomorphism and \( \mathcal{F} \) is an oriented foliation on \( \Sigma \), we obtain a map \( h^F \colon \mathbb{T}^2 \to \mathbb{R}^2 \setminus \{ 0 \} \) by pushing forward under \( F \) any vector field directing \( \mathcal{F} \). The map \( h^F_* \colon \pi_1(\mathbb{T}^2) \to \pi_1(S^1) \) is not an invariant of \( \mathcal{F} \). However, we can still say something. Let \( G \colon \Sigma \to \mathbb{T}^2 \) be another orientation-preserving diffeomorphism.
    \begin{claim}
        \( h^F_* = 0 \iff h^G_* = 0 \).
    \end{claim}
    \begin{proof}
        Let \( Y \) be a vector field directing \( \mathcal{F} \) on \( \Sigma \). It is well known from classical low-dimensional topology that any diffeomorphism of \( \mathbb{T}^2 \) is isotopic to a linear one. Thus, there exists \( A \in \operatorname{SL}_2 \mathbb{Z} \) such that \( FG^{-1} \) is isotopic to \( A \), the latter viewed as a diffeomorphism of the torus.
 Consider now the diagram
    \begin{equation*}
        \begin{tikzcd}
            \mathbb{T}^2 \arrow[d, "G_*Y", swap] \arrow[dd, bend right=80, looseness=2, "h^G", swap] & \Sigma \arrow[l, "G", swap] \arrow[d, "Y"] \arrow[r, "F"] & \mathbb{T}^2 \arrow[d, "F_*Y"] \arrow[dd, bend left=80, looseness=2, "h^F"] \\
            \mathbb{T}^2 \times (\mathbb{R}^2 \setminus \{ 0 \}) \arrow[d, "\mathrm{pr}_2", swap] & T \Sigma_0 \arrow[l, "dG", swap] \arrow[r, "dF"] & \mathbb{T}^2 \times (\mathbb{R}^2 \setminus \{ 0 \}) \arrow[d, "\mathrm{pr}_2"] \\
            \mathbb{R}^2 \setminus \{ 0 \} \arrow[rr, "A"] & & \mathbb{R}^2 \setminus \{ 0 \}
        \end{tikzcd}
    \end{equation*}
    where \( T \Sigma_0 \) denotes \( T \Sigma \) without the zero section. Note that the top right and top left squares commute, whereas in general the bottom one does not. However, since \( FG^{-1} \) is isotopic to \( A \), the diagram
    \begin{equation*}
        \begin{tikzcd}
            \pi_1(\mathbb{T}^2) \arrow[d, "\pi_1(G_*Y)", swap] \arrow[dd, bend right=75, looseness=2, "\pi_1(h^G)", swap]
        & \pi_1(\Sigma) \arrow[l, "\pi_1(G)", swap] \arrow[d, "\pi_1(Y)"] \arrow[r, "\pi_1(F)"]
            & \pi_1(\mathbb{T}^2) \arrow[d, "\pi_1(F_*Y)"] \arrow[dd, bend left=75, looseness=2, "\pi_1(h^F)"] \\
    \pi_1 \big( \mathbb{T}^2 \times (\mathbb{R}^2 \setminus \{ 0 \}) \big) \arrow[d, "\pi_1(\mathrm{pr}_2)", swap]
        & \pi_1(T \Sigma_0) \arrow[l, "\pi_1(dG)", swap] \arrow[r, "\pi_1(dF)"]
        & \pi_1 \big( \mathbb{T}^2 \times (\mathbb{R}^2 \setminus \{ 0 \}) \big) \arrow[d, "\pi_1(\mathrm{pr}_2)"] \\
    \pi_1(\mathbb{R}^2 \setminus \{ 0 \}) \arrow[rr, "\pi_1(A)"]
        &
        & \pi_1(\mathbb{R}^2 \setminus \{ 0 \})
\end{tikzcd}
    \end{equation*}
    does commute, which implies that the leftmost curved arrow is the trivial map if and only if the rightmost curved one is, as the top and bottom arrows are all isomorphisms.
    \end{proof}
\end{remark}

We will be particularly interested in foliations defined by a closed form. A necessary condition for this is that all leaves be pairwise diffeomorphic. Indeed, suppose \( \mathcal{F} = \mathop{\mathrm{ker}} \alpha \) with \( d \alpha = 0 \) and pick a vector field \( Y \) such that \( \alpha(Y) \equiv 1 \). Then \( \mathcal{L}_Y \alpha = 0 \), which implies that the flow of \( Y \) preserves \( \mathcal{F} \). This shows that the diffeomorphism type of the leaves of \( \mathcal{F} \) is locally constant and we can conclude using the fact that \( \mathbb{T}^2 \) is connected.
A consequence of this is that \( \mathcal{F} \) does not admit any Reeb annuli (see \cref{fig:reeb}), which means that, up to a diffeomorphism of the torus, it is the suspension of some orientation-preserving diffeomorphism \( \varphi \) of the circle.

We can actually say more.
\begin{lemma}\label{lemma:fols_by_closed}
    Let \( \alpha \in \Omega^1(\mathbb{T}^2) \) be a nowhere vanishing closed \( 1 \)-form and let \( \mathcal{F} = \mathop{\mathrm{ker}} \alpha \) be the induced foliation. Then, \( \mathcal{F} \) is homotopic to a constant foliation. Equivalently, if \( X \colon \mathbb{T}^2 \to \mathbb{R}^2 \setminus \{ 0 \} \) is a vector field directing the foliation, the map
    \begin{equation*}
        X_* \colon \pi_1(\mathbb{T}^2) \to \pi_1(S^1)
    \end{equation*}
    is trivial.
\end{lemma}
\begin{remark}
    When we say that a foliation is homotopic to another one, we mean homotopic \textit{through foliations}, that is no singularities are allowed.
\end{remark}
\begin{proof}
    As we remarked above, there exist an orientation-preserving diffeomorphism \( \varphi \colon S^1 \to S^1 \) and a diffeomorphism \( f \colon M_{\varphi} \to \mathbb{T}^2 \) such that \( f\big( \operatorname{Susp}(\varphi) \big) = \mathcal{F} \).
    Note that the mapping torus \( M_\varphi \) of \( \varphi \) comes with a well-defined projection \( q \colon M_{\varphi} \to S^1 \) induced by
    \begin{equation*}
        \begin{aligned}
            \operatorname{pr}_2 \colon S^1& \times [0, 1] \to S^1 = \mathbb{R} / \mathbb{Z} \\
            &(x, s) \mapsto s + \mathbb{Z},
        \end{aligned}
    \end{equation*}
    which makes it into a fiber bundle over \( S^1 \). Thus, \( M_{\varphi} \) carries the foliation \( \mathcal{G} \) by the fibers \( \{ q^{-1}(x) \}_{x \in S^1} \). Note that \( \mathcal{G} \) is transverse to \( \operatorname{Susp}(\varphi) \), so the claim follows if we can show that \( \mathcal{G} \) is homotopically trivial, meaning it induces the trivial map \( \pi_1(\mathbb{T}^2) \to \pi_1(S^1) \). By \cref{rmk:trivialization}, the vanishing of this map does not depend on our choice of identification \( M_{\varphi} \cong \mathbb{T}^2 \). Let \( \{ \varphi_s \}_{s \in [0, 1]} \) be an isotopy of diffeomorphisms of \( S^1 \) connecting \( \operatorname{id}_{S^1} \) and \( \varphi \) and consider the map
    \begin{equation*}
        \begin{aligned}
            G \colon S^1 &\times [0, 1] \to S^1 \times [0, 1] \\
            &(x, s) \mapsto \big( \varphi_{1-s}(x), s \big).
        \end{aligned}
    \end{equation*}
    It is straightforward to see that it induces an orientation-preserving diffeomorphism \( \bar{G} \colon \mathbb{T}^2 = S^1 \times S^1 \to M_{\varphi} \). Thus, we have to show that the foliation of \( \mathbb{T}^2 \) by the level sets of the map \( q \bar{G} \colon \mathbb{T}^2 \to S^1 \) is homotopically trivial. This is immediate, since \( q \bar{G} \) is simply the projection onto the second factor, giving the constant foliation of \( \mathbb{T}^2 \) by \( S^1 \times \{ p \} \) with \( p \in S^1 \).
\end{proof}

We end the section by briefly discussing pairs of transverse foliations without parallel compact leaves (see \cref{def:parallel_compact}).

It is easy to see that if a foliation \( \mathcal{F} \) admits a Reeb component, then any foliation \( \mathcal{G} \) transverse to \( \mathcal{F} \) admits a closed leaf in the interior of the Reeb component. In particular, \( \mathcal{F} \) and \( \mathcal{G} \) have parallel compact leaves. The converse is clearly false.

Pairs of transverse foliations on the torus without parallel compact leaves admit a nice normal form.
\begin{theorem}[{\cite[Thm.\ 5.1]{bonatti2016}}]\label{thm:affine_separating}
    Let \( \mathcal{F}, \mathcal{G} \) be two transverse foliations on \( \mathbb{T}^2 \) without parallel compact leaves. Then, there exist affine (i.e.\ constant) foliations \( \mathcal{H}, \mathcal{I} \) and a diffeomorphism \( \theta \colon \mathbb{T}^2 \to \mathbb{T}^2 \) such that the foliations \( \theta(\mathcal{F}), \theta(\mathcal{G}), \mathcal{H} \) and \( \mathcal{I} \) are pairwise transverse.
\end{theorem}

The theorem says that two transverse foliations without parallel compact leaves can be separated (up to changing them by a diffeomorphism of the torus) by constant foliations (see \cref{fig:affine_separating}).

\begin{figure}[h]
    \centering
    \includegraphics[scale=1]{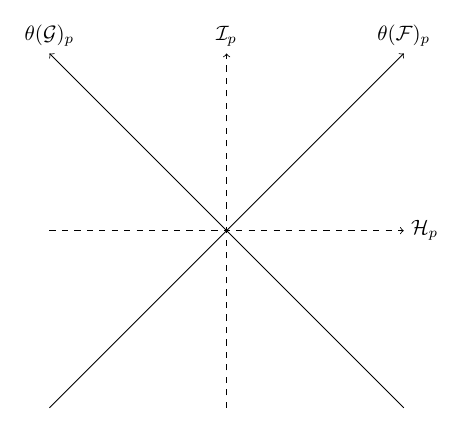}
    \caption{The constant foliations \( \mathcal{H} \) and \( \mathcal{I} \) divide the tangent space \( T_p \mathbb{T}^2 \) into four quadrants. \( \theta(\mathcal{F}) \) and \( \theta(\mathcal{G}) \) lie in different quadrants.}
    \label{fig:affine_separating}
\end{figure}

\section{Proof of \texorpdfstring{\cref{thm:obstruction}}{Theorem 1.2}}
\label{sec:main_proof}

Throughout this section, \( \Sigma \subseteq M \) will be an embedded torus transverse to the flow \( \varphi_t \colon M \to M \). Recall that, according to \cref{def:pre-lagrangian}, \( \Sigma \) is pre-Lagrangian if and only if it can embedded in \( \mathbb{R} \times M \) as a graph of a function which is Lagrangian with respect to the symplectic structure induced by some Anosov-Liouville pair supporting \( \varphi_t \).

\begin{remark}
    The function \( f \) in \cref{def:pre-lagrangian} can be assumed without loss of generality to be constantly equal to \( 0 \). Indeed, suppose \( \Sigma \) is pre-Lagrangian for some choice of \( f \in C^{\infty}(\Sigma) \) and \( (\alpha_+, \alpha_-) \). This means that the \( 1 \)-form \( e^f \alpha_+ |_{\Sigma} + e^{-f} \alpha_-|_{\Sigma} \) is closed. Extend now \( f \) to a function (which we still call \( f \)) on \( M \) and note that \( (e^f \alpha_+, e^{-f} \alpha_-) \) is still Anosov-Liouville for \( \varphi_t \). Thus, \( \{ 0 \} \times \Sigma \subseteq \mathbb{R} \times M \) is Lagrangian with respect to the symplectic structure induced by \( (e^f \alpha_+, e^{-f} \alpha_-) \).
\end{remark}

\begin{example}\label{ex:suspension}
    Consider the suspension of a linear hyperbolic automorphism of the torus. It is straightforward to see that every \( \mathbb{T}^2 \)-fiber is pre-Lagrangian (with the Anosov-Liouville pair being the one described in \cref{ex:examples} \ref{itm:ii}). Actually more is true, namely the fact that the Reeb vector fields of the contact forms described in the example are tangent to every fiber.
\end{example}

In view of the above remark, we will look for an Anosov-Liouville pair \( (\alpha_+, \alpha_-) \) such that \( \alpha_+|_{\Sigma} + \alpha_-|_{\Sigma} \) is closed. Note that \( \alpha_+|_{\Sigma} + \alpha_-|_{\Sigma} \) is a nonvanishing \( 1 \)-form and hence it defines a foliation on \( \Sigma \).

Let now \( \alpha_u \) and \( \alpha_s \) be \( C^1 \) \( 1 \)-forms as in \cref{subsec:al}. In particular, \( \mathop{\mathrm{ker}} \alpha_s = E^{wu} \) and \( \mathop{\mathrm{ker}} \alpha_u = E^{ws} \). Note that the foliations \( \mathcal{F}^{ws}_{\Sigma} \) and \( \mathcal{F}^{wu}_{\Sigma} \) on \( \Sigma \) are defined by \( \alpha_u|_{\Sigma} \) and \( \alpha_s|_{\Sigma} \) respectively. We are now ready to prove \cref{thm:obstruction}.

\begin{proof}[Proof of \cref{thm:obstruction}]
    By assumption, there exists an Anosov-Liouville pair \( (\alpha_+, \alpha_-) \) supporting \( \varphi_t \) such that \( d (\alpha_+ + \alpha_-)|_{\Sigma} = 0 \). Identifying \( \Sigma \) with \( \mathbb{T}^2 \), \cref{lemma:fols_by_closed} shows that the induced map \( \pi_1(\mathbb{T}^2) \to \pi_1(S^1) \) is trivial. Moreover, by \cref{rmk:trivialization}, this does not depend on the choice of identification \( \Sigma \cong \mathbb{T}^2 \).

    Now, by \cref{thm:contractible} the space of Anosov-Liouville pairs supporting \( \varphi_t \) is contractible. In particular, we can homotope the pair \( (\alpha_+, \alpha_-) \) to the standard pair \( (\alpha_u - \alpha_s, \alpha_u + \alpha_s) \) through Anosov-Liouville pairs. Thus, we obtain a path of foliations going from \( \ker (\alpha_+ + \alpha_-)|_{\Sigma} \) to \( \ker 2 \alpha_u|_{\Sigma} = \ker \alpha_u |_{\Sigma} \). This shows that \( \mathcal{F}^{ws}_{\Sigma} \) is trivial on \( \pi_1 \). The statement for \( \mathcal{F}^{wu}_{\Sigma} \) follows immediately, since \( \mathcal{F}^{ws}_{\Sigma} \pitchfork \mathcal{F}^{wu}_{\Sigma} \).
\end{proof}

The obstruction to being pre-Lagrangian provided by \cref{thm:obstruction} is not optimal. Indeed, we have already seen a necessary condition that is stronger, namely the fact that the leaves of the foliation defined by \( (\alpha_+ + \alpha_{-})|_{\Sigma} \) must be pairwise diffeomorphic. Let us consider a torus foliated as in the picture below and
\begin{figure}[ht]
    \centering
    \includegraphics[scale=1]{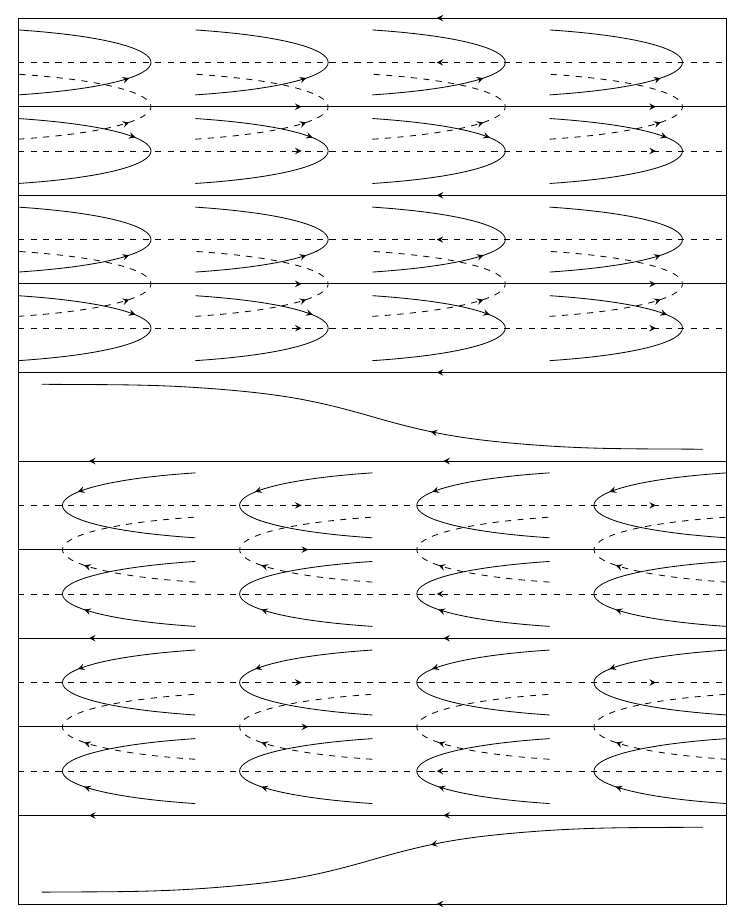}
    \caption{The black foliation is trivial on \( \pi_1 \). The dashed foliation is transverse to it (we have not drawn it in its entirety in order not to clutter the picture).}
    \label{fig:counterexample}
\end{figure}
suppose the black foliation is \( \mathcal{F}^{ws}_{\Sigma} \). By inspection, one sees that it is trivial on \( \pi_1 \) (the winding of the top Reeb annuli is \enquote{undone} by the bottom ones). Now, any foliation (the dashed one) that is transverse to it is going to have a number of Reeb annuli as well. This follows from observing that the dashed foliation admits closed leaves in the interior of every Reeb annulus of \( \mathcal{F}^{ws}_{\Sigma} \) by the Poincaré-Bendixson theorem and by inspecting their orientations. Let now \( (\alpha_+, \alpha_-) \) be an Anosov-Liouville pair and suppose that the dashed foliation in \cref{fig:counterexample} is induced by, say, \( \alpha_+|_{\Sigma} \) (note that the foliation \( \ker \alpha_+|_{\Sigma} \) is always transverse to \( \mathcal{F}^{ws}_{\Sigma} \), as the proof of \cite[Prop.\ 2.2.6]{eliashberg1998} shows). We can then consider the foliation defined by \( (\alpha_+ + \alpha_-)|_{\Sigma} \) and note that it is transverse to the dashed one. Thus, by the same reasoning, it is going to admit Reeb annuli. This shows that this torus cannot be pre-Lagrangian, despite \( \mathcal{F}^{ws}_{\Sigma} \) being trivial on \( \pi_1 \). The reason why we drew eight Reeb annuli in the black foliation is because we need a sufficient number of them to get this \enquote{cascading} effect.

Now, does there exist an Anosov flow that admits a transverse torus foliated as in \cref{fig:counterexample}? This question is answered in the affirmative by work of Beguin, Bonatti and Yu \cite{beguin2017}, which provides a very general procedure to construct Anosov flows starting from basic building blocks which the authors call \textit{hyperbolic plugs}. We highlight here the main points and refer the readert to \cite{beguin2017} for a thorough treatment.
A \textit{hyperbolic plug} \( (U, X) \) is a compact three-manifold \( U \) with boundary together with a vector field \( X \) such that:
\begin{enumerate}
    \item X is transverse to \( \partial U \) and
    \item the maximal \( X \)-invariant set \( \Lambda = \bigcap_{t \in \mathbb{R}} \varphi_t(U) \) (here \( \varphi_t \) is the flow of \( X \)) is hyperbolic with one dimensional stable and unstable bundles.
\end{enumerate}
Note that \( \partial U \) is necessarily a union of tori and that, in general, \( \partial U = \partial U_+ \sqcup \partial U_- \), where \( X \) is outward-pointing (resp.\ inward-pointing) on \( \partial U_+ \) (resp.\ \( \partial U_- \)).
A hyperbolic plug is an \textit{attractor} (resp.\ \textit{repeller}) if \( \partial U_+ = \emptyset \) (resp.\ \( \partial U_- = \emptyset \)).
In \cite{beguin2017} the authors show that, under some conditions, one can glue hyperbolic plugs together to construct Anosov flows. The Franks-Williams flow \cite{franks1980}, for instance, is obtained by suitably gluing together an attractor and a repeller with boundary foliations induced by the stable and unstable bundles as in \cref{fig:franks-williams}.
We are going to use two facts proven in \cite{beguin2017}:
\begin{itemize}
    \item \cite[Theorem 1.10]{beguin2017} says that for any Morse-Smale foliation \( \mathcal{F} \) on \( \mathbb{T}^2 \), there exists an attractor \( (U, X) \) and a homeomorphism \( h \colon \partial U \to \mathbb{T}^2 \) such that \( h_*(\mathcal{F}_X^s) = \mathcal{F} \), where \( \mathcal{F}_X^s \) denotes the foliation induced on \( \partial U \) by the stable bundle of \( X \). In other words, we can realize any Morse-Smale foliation on the \( 2 \)-torus as the entrance foliation of a hyperbolic attracting plug;
    \item a special case of \cite[Theorem 1.12]{beguin2017} says that (up to topological equivalence) we can embed any hyperbolic attracting plug \( (U, X) \) in an Anosov flow on a closed three-manifold. More precisely, there exists a closed three-manifold \( M \) carrying an Anosov vector field \( Y \) and an embedding \( \theta \colon U \to M \) such that \( \theta_* X = Y \).
\end{itemize}
These two facts together immediately imply that we can find an Anosov flow admitting a transverse torus foliated by, say, the stable bundle of the flow exactly as in \cref{fig:counterexample}.

\section{Examples of pre-Lagrangians}
\label{sec:examples}
In the previous sections, we have shown that \enquote{most} tori transverse to an Anosov flow are not pre-Lagrangian by providing a topological obstruction. So far, the only pre-Lagrangian tori that we have encountered are the \( \mathbb{T}^2 \)-fibers of suspension flows (cf.\ \cref{ex:suspension}). The goal of this section is to show that, although pathological cases such as the one in \cref{fig:counterexample} rule out a converse to \cref{thm:obstruction}, the induced foliations \( \mathcal{F}_{\Sigma}^u, \mathcal{F}_{\Sigma}^s \) on a transverse torus \( \Sigma \) can sometimes indicate pre-Lagrangianness.

\medskip

Consider an Anosov flow \( \varphi_t \colon M \to M \) generated by a vector field \( X \) and admitting a transverse torus \( \Sigma \subseteq M \).

\begin{lemma}\label{lemma:c1open}
    The pre-Lagrangian condition is \( C^1 \)-open.
\end{lemma}
\begin{proof}
    Suppose \( \Sigma \) is pre-Lagrangian and let \( (\alpha_+, \alpha_-) \) be an Anosov-Liouville pair such that \( (\alpha_+ + \alpha_-)|_{\Sigma} \) is closed.
    \begin{claim}
        For any Anosov-Liouville pair \( (\beta_+, \beta_-) \) such that \( (\beta_+ + \beta_-)|_{\Sigma} \) is \( C^1 \)-close to a \( 1 \)-form \( \beta \in \Omega^1(\Sigma) \), there exists an Anosov-Liouville pair \( (\tilde{\beta}_+, \tilde{\beta}_-) \) such that \( (\tilde{\beta}_+ + \tilde{\beta}_-)|_{\Sigma} = \beta \).
    \end{claim}
    \begin{proof}
        Let \( \sigma = \frac{1}{2} (\beta - \beta_+|_{\Sigma} - \beta_-|_{\Sigma}) \) and extend it in a \( C^1 \)-small fashion to \( M \) in such a way that \( \sigma(X) = 0 \). It is then easy to verify, using \cref{lemma:al_characterization}, that \( (\tilde{\beta}_+, \tilde{\beta}_-) = (\beta_+ + \sigma, \beta_- + \sigma) \) is an Anosov-Liouville pair (supporting \( \varphi_t \)) provided \( \sigma \) is \( C^1 \)-small enough. Moreover, by construction, \( (\tilde{\beta}_+ + \tilde{\beta}_-)|_{\Sigma} = \beta \).
    \end{proof}
    The lemma now follows directly from the claim, since for any transverse torus \( \Sigma' \) \( C^1 \)-close enough to \( \Sigma \), the form \( (\alpha_+ + \alpha_-)|_{\Sigma'} \) is \( C^1 \)-close to a closed one.

\end{proof}

In order to prove \cref{thm:existence}, we will need the following lemma.

\begin{lemma}\label{lemma:scaling}
    Let \( (\alpha_u, \alpha_s) \) be a defining pair for \( \varphi_t \) and \( f, g \colon \Sigma \to \mathbb{R}_{> 0} \) be positive functions. Then, there exists a defining pair \( (\tilde{\alpha}_u, \tilde{\alpha}_s) \) such that
    \begin{equation*}
        \begin{aligned}
            \tilde{\alpha}_u|_{\Sigma} &= f \alpha_u|_{\Sigma}, \\
            \tilde{\alpha}_s|_{\Sigma} &= g \alpha_s|_{\Sigma}. \\
        \end{aligned}
    \end{equation*}
\end{lemma}
\begin{proof}
    Note that if \( \lambda \colon M \to \mathbb{R}_{>0} \) is a positive function and \( \tilde{\alpha}_u = \lambda \alpha_u \),
    \begin{equation*}
        \mathcal{L}_X \tilde{\alpha}_u = \frac{X \lambda + \lambda r_u}{\lambda} \tilde{\alpha}_u = (X \log{\lambda} + r_u) \tilde{\alpha}_u.
    \end{equation*}
    The strategy will be to construct a positive function \( \lambda \) such that \( X \log{\lambda} + r_u \) is everywhere positive and \( \lambda|_{\Sigma} = f \). Let us identify a neighborhood of \( \Sigma \) with \( \Sigma \times (- 2\delta, 2\delta)_z \), with \( X \) being identified with \( \partial_z \). Let us impose the conditions
    \begin{equation*}
        \begin{cases}
            X \log \lambda + r_u = 1 & \text{near } \Sigma, \\
            \lambda|_\Sigma = f.
        \end{cases}
    \end{equation*}
    This forces
    \begin{equation*}
        \lambda(p, z) = f(p) e^{\int_0^z (1 - r_u(p)) \mathop{}\!\mathrm{d} s}
    \end{equation*}
    near \( \Sigma \). Thus, given \( 0 < \epsilon < \delta \), we define
    \begin{equation*}
        \lambda(p, z) =
        \begin{cases}
            C & \text{if } z \geq \delta, \\
            f(p) e^{\int_0^z (1 - r_u(p)) \mathop{}\!\mathrm{d} s} & \text{if } -\epsilon \leq z \leq \epsilon, \\
            c & \text{if } z \leq -\delta,
        \end{cases}
    \end{equation*}
    for \( 0 < c < C \) satisfying
    \begin{equation*}
        \begin{aligned}
            C & > \max_{p \in \Sigma} \lambda(p, \epsilon), \\
            c &< \min_{p \in \Sigma} \lambda(p, -\epsilon).
        \end{aligned}
    \end{equation*}
    The only thing left to do is to extend \( \lambda \) over \( \Sigma \times (\epsilon, \delta) \) and \( \Sigma \times (-\delta, -\epsilon) \) in such a way as to keep \( X \log{\lambda} + r_u \) positive. For fixed parameter \( p \) it is straightforward, while in order to guarantee differentiability we need a canonical way of interpolating. We could for instance exponentially interpolate, i.e.\ set
    \begin{equation*}
        \lambda(p, z) = \lambda(p, \epsilon)^{\frac{\delta - z}{\delta - \epsilon}} C^{\frac{z - \epsilon}{\delta - \epsilon}}, \quad z \in (\epsilon, \delta)
    \end{equation*}
    (and similarly on \( (-\delta, -\epsilon) \)) and then smooth the result. Since \( p \) varies in a compact space, these smoothings can be performed in such a way as to keep \( X \log{\lambda} + r_u \) positive.
    An analogous argument takes care of \( \alpha_s \). This finishes the proof.
\end{proof}

\begin{proof}[Proof of \cref{thm:existence}]
    Let \( (\alpha_u, \alpha_s) \) be a defining pair for \( \varphi_t \) and assume that the foliations \( \mathcal{F}^{u}_{\Sigma} \) and \( \mathcal{F}^s_{\Sigma} \) do not have parallel compact leaves. By \cref{thm:affine_separating}, they can be separated by two foliations that are defined by \textit{closed} \( 1 \)-forms. Thus, there exist positive functions \( f, g \colon \Sigma \to \mathbb{R}_{> 0} \) such that \( (f \alpha_u - g \alpha_s)|_{\Sigma} \) is closed.
    By \cref{lemma:scaling}, we can assume without loss of generality that \( (\alpha_u - \alpha_{s})|_{\Sigma} \) is closed. Let \( (\alpha_+, \alpha_-) = (\alpha_u - \alpha_s, \alpha_u + \alpha_s) \) be the associated standard Anosov-Liouville pair. Replacing it by \( (C \alpha_+, \frac{1}{C} \alpha_-) \), we can further assume that \( (\alpha_+ + \alpha_-)|_{\Sigma} \) is as \( C^1 \)-close to a closed form (i.e.\ \( C \alpha_+|_{\Sigma} \)) as we wish. Then, the proof of \cref{lemma:c1open} shows that \( \Sigma \) is pre-Lagrangian.
\end{proof}

\cref{thm:existence} for instance applies to the Anosov flow constructed by Bonatti and Langevin in \cite{bonatti1994}, which provided the first example of a transitive Anosov flow not equivalent to a suspension flow and admitting a transverse torus. By construction, the foliations induced on the torus by the weak stable and unstable bundles admit transverse closed leaves. This in particular implies that they do not have parallel compact leaves and hence the torus is pre-Lagrangian.

Other examples can be obtained by making use of the aforementioned machinery developed by Bonatti, Beguin and Yu. Mimicking the construction of the Franks-Williams flow, one could for example glue together a hyperbolic attractor \( (U, X) \) and a hyperbolic repeller \( (V, Y) \), both with torus boundary and with Reebless boundary foliations \( \mathcal{F}_U \) and \( \mathcal{F}_{V} \). Their existence is ensured by \cite[Theorem 1.10]{beguin2017}. If the gluing map \( \psi \colon \partial U \to \partial V \) is chosen in such a way that the foliations \( \psi(\mathcal{F}_{U}) \) and \( \mathcal{F}_V \) are transverse and do not admit parallel compact leaves, then the resulting torus in \( U \cup_{\psi} V \) is pre-Lagrangian.

\bibliographystyle{alpha}
\bibliography{bib}

\begin{thebibliography}{CLMM22}

\bibitem[BBY17]{beguin2017}
François Béguin, Christian Bonatti, and Bin Yu.
\newblock Building anosov flows on 3–manifolds.
\newblock {\em Geometry \& Topology}, 21(3):1837–1930, May 2017.

\bibitem[BL94]{bonatti1994}
Christian Bonatti and Remi Langevin.
\newblock Un exemple de flot d’anosov transitif transverse à un tore et non
  conjugué à une suspension.
\newblock {\em Ergodic Theory and Dynamical Systems}, 14(4):633–643, 1994.

\bibitem[Bru93]{brunella1993}
Marco Brunella.
\newblock {Separating the Basic Sets of a Nontransitive Anosov Flow}.
\newblock {\em Bulletin of the London Mathematical Society}, 25(5):487--490, 09
  1993.

\bibitem[BZ16]{bonatti2016}
Christian Bonatti and Jinhua Zhang.
\newblock Transverse foliations on the torus {$\mathbb{T}^2$} and partially
  hyperbolic diffeomorphisms on 3-manifolds, 2016.

\bibitem[CLMM22]{cieliebak2022}
Kai Cieliebak, Oleg Lazarev, Thomas Massoni, and Agustin Moreno.
\newblock Floer theory of anosov flows in dimension three, 2022.

\bibitem[ET98]{eliashberg1998}
Y.~Eliashberg and W.P. Thurston.
\newblock {\em Confoliations}.
\newblock University lecture series. American Mathematical Society, 1998.

\bibitem[FW80]{franks1980}
John Franks and Bob Williams.
\newblock Anomalous anosov flows.
\newblock In Zbigniew Nitecki and Clark Robinson, editors, {\em Global Theory
  of Dynamical Systems}, page 158–174, Berlin, Heidelberg, 1980. Springer
  Berlin Heidelberg.

\bibitem[Gei95]{geiges1995}
Hansjörg Geiges.
\newblock {Examples of Symplectic 4-Manifolds with Disconnected Boundary of
  Contact Type}.
\newblock {\em Bulletin of the London Mathematical Society}, 27(3):278--280, 05
  1995.

\bibitem[Has94]{hasselblatt1994}
Boris Hasselblatt.
\newblock Regularity of the anosov splitting and of horospheric foliations.
\newblock {\em Ergodic Theory and Dynamical Systems}, 14(4):645–666, 1994.

\bibitem[HH81]{hector1981}
G.~Hector and U.~Hirsch.
\newblock {\em Introduction to the Geometry of Foliations}.
\newblock Number pt. 1-2 in Aspects of mathematics. Vieweg, 1981.

\bibitem[Hoz22]{hozoori2022}
Surena Hozoori.
\newblock Symplectic geometry of anosov flows in dimension 3 and bi-contact
  topology, 2022.

\bibitem[Mas23]{massoni2023}
Thomas Massoni.
\newblock Anosov flows and liouville pairs in dimension three, 2023.

\bibitem[McD91]{mcduff1991}
Dusa McDuff.
\newblock Symplectic manifolds with contact type boundaries.
\newblock {\em Inventiones mathematicae}, 103(1):651--671, 1991.

\bibitem[Mit95]{mitsumatsu1995}
Yoshihiko Mitsumatsu.
\newblock Anosov flows and non-stein symplectic manifolds.
\newblock {\em Annales de l'institut Fourier}, 45(5):1407--1421, 1995.

\bibitem[Sma67]{smale1967}
Stephen Smale.
\newblock {Differentiable dynamical systems}.
\newblock {\em Bulletin of the American Mathematical Society}, 73(6):747 --
  817, 1967.

\end{thebibliography}

\end{document}